\documentclass[11pt,a4paper]{amsart}

\usepackage[T1]{fontenc}
\usepackage{lmodern}
\usepackage[margin=28mm]{geometry}
\usepackage{mathtools}
\usepackage{amssymb}
\usepackage{mathrsfs}
\usepackage{enumitem}
\usepackage{array}
\usepackage[sort,compress]{cite}
\usepackage{xcolor}
\usepackage{microtype}
\usepackage{placeins}
\usepackage[hypertexnames=false]{hyperref}
\usepackage{xurl}
\hypersetup{
  hidelinks,
  pdftitle={Torsion and Positive Rank in an Elliptic Family Arising from Cubic 2-Cycles},
  pdfauthor={Sompong Chuysurichay, Sawian Jaidee, and Chatchawan Panraksa},
  pdfsubject={Arithmetic dynamics and elliptic curves},
  pdfkeywords={arithmetic dynamics, elliptic curves, rational periodic points, Prym varieties, torsion subgroups, cubic polynomials}
}
\makeatletter
\g@addto@macro{\UrlBreaks}{\do\/\do\_\do\-}
\makeatother

\newtheorem{theorem}{Theorem}[section]
\newtheorem{proposition}[theorem]{Proposition}
\newtheorem{lemma}[theorem]{Lemma}
\newtheorem{corollary}[theorem]{Corollary}
\theoremstyle{definition}
\newtheorem{example}[theorem]{Example}
\theoremstyle{remark}
\newtheorem{remark}[theorem]{Remark}
\numberwithin{equation}{section}
\newcommand{\Q}{\mathbb{Q}}
\newcommand{\Z}{\mathbb{Z}}
\newcommand{\mathcalE}{\mathcal{E}}

\newcommand{\rank}{\operatorname{rank}}

\title[Torsion and Rank in a Cubic 2-Cycle Family]{Torsion and Positive Rank in an Elliptic Family Arising from Cubic 2-Cycles}

\author[S. Chuysurichay]{Sompong Chuysurichay}
\address{Division of Computational Science, Faculty of Science, Prince of Songkla University, Songkhla 90110, Thailand}
\email{sompong.c@psu.ac.th}

\author[S. Jaidee]{Sawian Jaidee}
\address{Department of Mathematics, Faculty of Science, Khon Kaen University, Khon Kaen 40002, Thailand}
\email{jsawia@kku.ac.th}

\author[C. Panraksa]{Chatchawan Panraksa}
\address{Applied Mathematics Program, Mahidol University International College, Nakhon Pathom 73170, Thailand}
\email{chatchawan.pan@mahidol.ac.th}
\thanks{Corresponding author: Chatchawan Panraksa.}

\date{}

\begin{document}

\begin{abstract}
We study the elliptic family \(\mathcal{E}_t:Y^2=X^3+4X^2+16t^2\) arising
from rational \(2\)-cycles of \(x^3+bx+a\). For every
\(t\in\mathbb{Q}^\times\), we prove that \(P_t=(0,-4t)\) has infinite
order, so every nonsingular rational fiber has positive rank. We determine
its rational torsion subgroup: it is cyclic of order \(2\) precisely when \(t=u(u^2+4)/4\) for some
\(u\in\mathbb{Q}^\times\), and is trivial otherwise. No such fiber admits
a rational \(3\)- or \(5\)-isogeny. An explicit birational dictionary
then shows that, for every fixed \(a\in\mathbb{Q}^\times\), infinitely many
\(b\in\mathbb{Q}\) yield a rational \(2\)-cycle of \(x^3+bx+a\).
The uniform non-torsion assertion follows from Nagell--Lutz integrality.
The torsion exclusions combine elementary \(2\)-descent with explicit
genus-\(3\) curves, an unconditional rank-zero Prym argument, and two-cover
descent with elliptic Chabauty. Exact Magma and SageMath certificates
accompany the computer-assisted steps.
\end{abstract}

\maketitle

\section{Introduction}\label{sec:introduction}

Rational periodic points of low-degree polynomial maps often lead to
Diophantine problems on curves of positive genus
\cite{morton1994rational,morton1995periodic,silverman2007dynamics,
poonen2012uniform,flynn1996cycles,stoll2006rational}. For cubic maps, the
period-$2$ locus admits elliptic fibrations, allowing the arithmetic of
elliptic curves to be used in the study of rational cycles. We consider
the problem obtained by fixing the constant term: for a given
$a\in\Q^\times$, which $b\in\Q$ make
\begin{equation}\label{eq:intro-fba}
f_{b,a}(x)=x^3+bx+a
\end{equation}
admit a rational point of exact period $2$? Throughout, a rational
$2$-cycle is an ordered pair $(x_1,x_2)\in\Q^2$ satisfying
\[
f_{b,a}(x_1)=x_2,\qquad f_{b,a}(x_2)=x_1,\qquad x_1\neq x_2.
\]
We prove that infinitely many such $b$ exist for every nonzero rational
$a$. The arithmetic result underlying this conclusion is uniform over
all nonsingular rational fibers of the associated elliptic family.

The family \eqref{eq:intro-fba} is conjugate over $\Q$ to a
square-coefficient slice of cubic normal form~II
\cite[Definition~4.1 and Proposition~4.2]{benedetto2009computing}.
The period-$2$ equations on this slice yield
\[
\mathcalE_t:\quad Y^2=X^3+4X^2+16t^2.
\]
We use $t$ for the parameter in this elliptic family and write
$E_a=\mathcalE_a$ for the fiber attached to the fixed constant term $a$.
An explicit dictionary relates its rational points, away from specified
exceptional loci, to the cycles of \eqref{eq:intro-fba}.

Ingram \cite[Section~4]{ingram2012cubic} studies cubic polynomials with
periodic cycles of a specified multiplier. His period-$2$ curves
$X_0(2)$ and $X_1(2)$ are non-isotrivial elliptic curves over
$\mathbb C(\lambda)$, related by a $2$-isogeny. He determines the
Mordell--Weil group of $X_0(2)$ over the extensions
$\mathbb C(\lambda^{1/n})$; both curves have rank $0$ over
$\mathbb C(\lambda)$ and rank $3$ over
$\mathbb C(\lambda^{1/12})$. Here these symbols refer to Ingram's
dynamical curves. Thus the elliptic-curve description of cubic
$2$-cycles has an established precedent.

Our fixed-constant-term fibration differs from Ingram's multiplier
fibration. The distinction is also arithmetic: we exhibit a section
whose specialization has infinite order on every nonsingular rational
fiber. Generic non-torsion alone would allow exceptional specializations
and would not prove this assertion. We also determine the rational
torsion on every nonsingular rational fiber and exclude rational $3$- and $5$-isogenies.
The non-torsion statement gives the fixed-$a$ dynamical conclusion above,
which does not follow from the results on the multiplier fibration.

The torsion analysis uses Mazur's classification over $\Q$
\cite{mazur1977modular,mazur1978rational}. Merel's uniform boundedness
theorem \cite{merel1996bornes} and the small-degree results of
\cite{derickx2023torsion} place this calculation in the broader study of
torsion over number fields. We use the standard elliptic-curve
arithmetic in \cite{silverman2009arithmetic,washington2008elliptic}.

The following theorem gives the uniform arithmetic results and their
dynamical consequence.

\begin{theorem}[Main theorem]\label{thm:main-family}
Let
\[
\mathcalE_t:\quad Y^2=X^3+4X^2+16t^2,
\qquad t\in\Q^\times.
\]
Then:
\begin{enumerate}[label=\textup{(\roman*)},leftmargin=*]
\item \(P_t=(0,-4t)\) has infinite order, and hence
      \(\rank\mathcalE_t(\Q)\geq1\);
\item
\[
\mathcalE_t(\Q)_{\mathrm{tors}}\cong
\begin{cases}
\Z/2\Z,&t=u(u^2+4)/4\text{ for a unique }u\in\Q^\times,\\
0,&\text{otherwise};
\end{cases}
\]
\item \(\mathcalE_t\) admits no rational \(3\)- or \(5\)-isogeny.
\end{enumerate}
Consequently, for every \(a\in\Q^\times\), there are infinitely many
\(b\in\Q\) for which \(x^3+bx+a\) admits a rational \(2\)-cycle.
\end{theorem}

Sections~\ref{sec:background} and~\ref{sec:dictionary} derive the
specialized curve and the dictionary. The key step in
Section~\ref{sec:rank} is a Nagell--Lutz argument: the sixth multiple
of $P_t$ has a nonintegral $x$-coordinate on an integral model of each
nonzero rational fiber. The multiples of $P_t$ then give infinitely many distinct
parameters $b$. This proof of the rank and dynamical assertions is
independent of the later torsion exclusions.

For the torsion classification, Section~\ref{sec:elliptic-family}
identifies the rational $2$-torsion locus. The $3$-division polynomial
has no rational root, giving the $3$-isogeny exclusion in
Section~\ref{sec:torsion-3}. At level $5$, the normalization of the
modular fiber product admits a smooth plane quartic model. Its rank-zero
Prym has rational torsion of exponent dividing $2$. This forces the
rational points of the quartic into the fixed locus of
an involution and excludes rational $5$-isogenies
(Section~\ref{sec:torsion-5}). At level $7$, Tate normal form imposes a
fourth-power condition on $c_4$. Two-cover descent, norm maps and
elliptic Chabauty determine the rational points on the resulting genus-$3$
curve; a $7$-adic valuation then excludes the noncuspidal parameters
(Section~\ref{sec:torsion-7}). Finally, Section~\ref{sec:torsion-2primary} reduces the exclusion of
orders $4$ and $8$ to a genus-one quartic. The elementary $2$-descent in
Appendix~\ref{app:C4-descent} determines its rational points and completes
the classification. Appendix~\ref{app:computational} records the
computational certificates and their proof hypotheses.

\section{Cubic Normal Forms and the Specialized Period-2 Slice}\label{sec:background}

\subsection{Normal Form II and its period-2 curve}\label{subsec:background-normal-forms}

Over $\Q$, every cubic polynomial is linearly conjugate to one of the two standard normal forms
\begin{equation}\label{eq:background-normal-forms}
\text{(NF I)}\quad \phi(z)=uz^3+vz,\qquad
\text{(NF II)}\quad \psi(z)=Az^3+Bz+1,
\end{equation}
with $u,A\in\Q^\times$ and $v,B\in\Q$; see
\cite[Definition~4.1 and Proposition~4.2]{benedetto2009computing}.
The first form is odd, while the second has constant term $1$.
For fixed $A$, the period-$2$ equations of NF~II lead to the elliptic
curve derived below.

For fixed $a\neq0$, the change of variable $x=az$ takes
$f_{b,a}$ to the slice $A=a^2$ of NF~II. We first derive the
period-$2$ curve for general $A$, then specialize the resulting formulas.

Let
\[
\psi(z)=Az^3+Bz+1,\qquad A\in\Q^\times,\quad B\in\Q.
\]
Suppose that $(z_1,z_2)\in\Q^2$ is a rational $2$-cycle for $\psi$, so
\[
\psi(z_1)=z_2,\qquad \psi(z_2)=z_1,\qquad z_1\neq z_2.
\]
Introduce the symmetric parameters $s,p$ and the alternating parameter $w$ by
\[
s:=z_1+z_2,\qquad p:=z_1z_2,\qquad w:=z_1-z_2,
\]
so that $w^2=s^2-4p$.

Subtracting the two cycle equations and dividing by $z_1-z_2\neq 0$ gives
\begin{equation}\label{eq:nfii-symmetric}
A(z_1^2+z_1z_2+z_2^2)+B=-1.
\end{equation}
Equivalently,
\begin{equation}\label{eq:nfii-symmetric-sp}
A(s^2-p)+B=-1.
\end{equation}
Adding the cycle equations yields
\[
s=A(z_1^3+z_2^3)+Bs+2=A\,s(s^2-3p)+Bs+2.
\]
The case $s=0$ is impossible, since the displayed equation would then read $0=2$. Therefore $s\neq 0$, and dividing by $s$ gives
\begin{equation}\label{eq:nfii-sum}
A(s^2-3p)+B=1-\frac{2}{s}.
\end{equation}
Subtracting \eqref{eq:nfii-symmetric-sp} from \eqref{eq:nfii-sum} gives
\begin{equation}\label{eq:nfii-p}
p=\frac{1-s}{As}.
\end{equation}
Using $w^2=s^2-4p$, we obtain
\begin{equation}\label{eq:nfii-w}
w^2=\frac{As^3+4s-4}{As}.
\end{equation}

Now set
\begin{equation}\label{eq:nfii-XY}
X:=\frac{1}{s},\qquad Y:=AXw.
\end{equation}
Then \eqref{eq:nfii-w} becomes
\begin{equation}\label{eq:nfii-curve}
E^{(A)}:\quad Y^2=A(-4X^3+4X^2+A).
\end{equation}
The cubic polynomial on the right-hand side of \eqref{eq:nfii-curve} has
discriminant
\[
\operatorname{disc}_X\!\bigl(A(-4X^3+4X^2+A)\bigr)
=-16A^5(27A+16),
\]
so \eqref{eq:nfii-curve} defines an elliptic curve for $A\in\Q^\times$ with $27A+16\neq 0$. Thus the fixed-coefficient period-$2$ equations give the elliptic curve \(E^{(A)}\); the singular value $A=-16/27$ is excluded in the proposition below.

\begin{proposition}\label{prop:nfii-period-two}
Fix $A\in\Q^\times$ with $27A+16\neq 0$. Then the set of triples
$(B,z_1,z_2)\in\Q^3$ satisfying $z_1\neq z_2$ and
\[
Az_1^3+Bz_1+1=z_2,\qquad Az_2^3+Bz_2+1=z_1
\]
is in bijection with points $(X,Y)\in E^{(A)}(\Q)$ satisfying
$X\neq 0$ and $Y\neq 0$. The maps are
\begin{equation}\label{eq:nfii-forward}
(B,z_1,z_2)\longmapsto \left(\frac{1}{z_1+z_2},\,A\,\frac{z_1-z_2}{z_1+z_2}\right)
\end{equation}
and
\begin{equation}\label{eq:nfii-backward}
(X,Y)\longmapsto
\left(X-2-\frac{A}{X^2},\,\frac{1}{2X}+\frac{Y}{2AX},\,\frac{1}{2X}-\frac{Y}{2AX}\right).
\end{equation}
The corresponding parameter $B$ is recovered from $X$ by
\begin{equation}\label{eq:nfii-B}
B=X-2-\frac{A}{X^2}.
\end{equation}
\end{proposition}

\begin{proof}
The forward map is exactly the construction above. Conversely, let $(X,Y)\in E^{(A)}(\Q)$ with $X,Y\neq 0$ and set
\[
s=\frac{1}{X},\qquad w=\frac{Y}{AX}.
\]
Since $(X,Y)$ satisfies \eqref{eq:nfii-curve}, reversing the calculation above shows that \eqref{eq:nfii-p} and \eqref{eq:nfii-w} hold with
\[
p=\frac{X-1}{A}.
\]
Defining
\[
z_1=\frac{s+w}{2},\qquad z_2=\frac{s-w}{2},
\]
and taking $B$ from \eqref{eq:nfii-B} gives \eqref{eq:nfii-backward}.
Since $w\neq0$, the two points $z_1$ and $z_2$ are distinct. Moreover, \eqref{eq:nfii-symmetric-sp} yields \eqref{eq:nfii-B}, and substitution into $\psi(z)=Az^3+Bz+1$ shows that $\psi(z_1)=z_2$ and $\psi(z_2)=z_1$. The two constructions are inverse to each other.
\end{proof}

\subsection{Conjugating the cubic family to the square-coefficient slice}\label{subsec:background-conjugacy}

We now return to the family
\begin{equation}\label{eq:background-fba}
f_{b,a}(x)=x^3+bx+a,\qquad a,b\in\Q.
\end{equation}

\begin{proposition}\label{prop:conjugacy}
For each $a\in\Q^\times$, the map $f_{b,a}$ is linearly conjugate over $\Q$ to
\[
\psi_{a,b}(z)=a^2z^3+bz+1,
\]
which is in NF II with parameters $A=a^2$ and $B=b$.
\end{proposition}

\begin{proof}
Put $M(z)=az$. Then
\(M^{-1}\circ f_{b,a}\circ M(z)=f_{b,a}(az)/a=a^2z^3+bz+1\).
\end{proof}

From now on we take $a\neq0$. Because the conjugacy is defined over $\Q$,
it preserves rational periodic points and exact periods. Thus rational
$2$-cycles of $f_{b,a}$ are governed by the curve
\eqref{eq:nfii-curve} with $A=a^2$.

For later use it is convenient to put this specialized curve into the simpler model
\begin{equation}\label{eq:Ea-model}
E_a:\quad Y^2=X^3+4X^2+16a^2.
\end{equation}
If
\[
E^{(a^2)}:\quad Y_0^2=a^2(-4X_0^3+4X_0^2+a^2),
\]
then the change of variables
\begin{equation}\label{eq:Ea-change}
X=-4X_0,\qquad Y=\frac{4}{a}Y_0
\end{equation}
identifies $E^{(a^2)}$ with $E_a$. The discriminant of \eqref{eq:Ea-model} is
\[
\Delta(E_a)=-4096a^2(27a^2+16),
\]
which is nonzero for every $a\in\Q^\times$.

\section{Birational Dictionary for Rational 2-Cycles}\label{sec:dictionary}

Combining Proposition~\ref{prop:nfii-period-two} with Proposition~\ref{prop:conjugacy} gives an explicit dictionary between rational points on $E_a$ off the exceptional locus $\{\mathcal O,X=0,Y=0\}$ and rational $2$-cycles of maps in the one-parameter family \eqref{eq:background-fba} with fixed nonzero $a$.

\begin{theorem}[Birational dictionary]\label{thm:dictionary}
Fix $a\in\Q^\times$. There is a bijection between
\begin{enumerate}[label=\textup{(\arabic*)},leftmargin=*]
\item rational points $Q\in E_a(\Q)\setminus\{\mathcal O\}$ that can be written $Q=(X,Y)$ with $X\neq 0$ and $Y\neq 0$, and
\item ordered triples $(b,x_1,x_2)\in\Q^3$ such that $x_1\neq x_2$ and
\[
f_{b,a}(x_1)=x_2,\qquad f_{b,a}(x_2)=x_1.
\]
\end{enumerate}
The forward map is
\begin{equation}\label{eq:dictionary-forward}
b=-\frac{X}{4}-2-\frac{16a^2}{X^2},\qquad
x_1=\frac{-4a-Y}{2X},\qquad
x_2=\frac{-4a+Y}{2X}.
\end{equation}
The inverse map is
\begin{equation}\label{eq:dictionary-inverse}
(b,x_1,x_2)\longmapsto \left(X,Y\right)=\left(-\frac{4a}{x_1+x_2},\,\frac{4a(x_1-x_2)}{x_1+x_2}\right).
\end{equation}
\end{theorem}

\begin{proof}
Let $(X,Y)\in E_a(\Q)$ with $X,Y\neq 0$ and set
\[
X_0=-\frac{X}{4},\qquad Y_0=\frac{aY}{4}.
\]
By \eqref{eq:Ea-change}, the point $(X_0,Y_0)$ lies on $E^{(a^2)}(\Q)$ and still satisfies $X_0,Y_0\neq 0$. Applying Proposition~\ref{prop:nfii-period-two} with $A=a^2$ gives a rational $2$-cycle
\[
z_1=\frac{1}{2X_0}+\frac{Y_0}{2a^2X_0},\qquad
z_2=\frac{1}{2X_0}-\frac{Y_0}{2a^2X_0}
\]
for the map $\psi_{a,b}(z)=a^2z^3+bz+1$, where
\[
b=X_0-2-\frac{a^2}{X_0^2}.
\]
Multiplying by $a$ and simplifying yields \eqref{eq:dictionary-forward}. Since $f_{b,a}=M\circ \psi_{a,b}\circ M^{-1}$ with $M(z)=az$, the pair $(x_1,x_2)=(az_1,az_2)$ is a rational $2$-cycle of $f_{b,a}$.

Conversely, suppose that $f_{b,a}$ admits a rational $2$-cycle $(x_1,x_2)$. First note that \(x_1+x_2\neq0\): if \(x_2=-x_1\), then adding the equations
\[
x_2=f_{b,a}(x_1),\qquad x_1=f_{b,a}(x_2)
\]
would give \(0=2a\), contradicting \(a\neq0\). Set
\[
z_1=\frac{x_1}{a},\qquad z_2=\frac{x_2}{a}.
\]
Then $(z_1,z_2)$ is a rational $2$-cycle for $\psi_{a,b}$. Proposition~\ref{prop:nfii-period-two} therefore gives a point $(X_0,Y_0)\in E^{(a^2)}(\Q)$ with
\[
X_0=\frac{1}{z_1+z_2}=\frac{a}{x_1+x_2},\qquad
Y_0=a^2\frac{z_1-z_2}{z_1+z_2}=a^2\frac{x_1-x_2}{x_1+x_2}.
\]
Applying \eqref{eq:Ea-change} gives \eqref{eq:dictionary-inverse}. The two constructions are inverse to each other.
\end{proof}

\begin{remark}\label{rem:dictionary-degeneracies}
Replacing $Y$ by $-Y$ in \eqref{eq:dictionary-forward} exchanges $x_1$
and $x_2$, so the two points $(X,\pm Y)$ represent the two orderings of
the same unordered cycle. If $Y=0$, the formulas instead give
$x_1=x_2=-2a/X$. The curve equation then gives
$f_{b,a}(x_1)=x_1$ and $f'_{b,a}(x_1)=-1$: these are the fixed points
arising from the degenerate period-$2$ equations.

The points with $X=0$ are exactly $(0,\pm4a)$ and are excluded from the
dictionary. No rational $2$-cycle has such an image, since
$X=-4a/(x_1+x_2)\neq0$ in \eqref{eq:dictionary-inverse}.
\end{remark}

\begin{example}\label{ex:dictionary-first-example}
The point $(-4,4a)$ lies on $E_a(\Q)$ for every $a\in\Q^\times$. It is in fact the double of the distinguished point \(P_a=(0,-4a)\): direct doubling gives \([2]P_a=(-4,4a)\). Applying \eqref{eq:dictionary-forward} gives
\[
b=-1-a^2,\qquad (x_1,x_2)=(a,0).
\]
A direct check shows that
\[
f_{-1-a^2,a}(a)=0,\qquad f_{-1-a^2,a}(0)=a,
\]
so this indeed yields a rational $2$-cycle of $f_{-1-a^2,a}$.
For example, when \(a=1\) this gives \(b=-2\) and the rational
\(2\)-cycle \((1,0)\) of \(x^3-2x+1\).
\end{example}

\section{Infinite Order of the Distinguished Section and Dynamical Consequences}\label{sec:rank}

The rank assertion follows from Nagell--Lutz integrality.

\begin{theorem}\label{thm:Pt-infinite}
For every $t\in\Q^\times$, the point $P_t=(0,-4t)$ has infinite order on
\begin{equation}\label{eq:Et-family}
\mathcalE_t:\quad Y^2=X^3+4X^2+16t^2.
\end{equation}
In particular, $\rank\mathcalE_t(\Q)\ge1$.
\end{theorem}

\begin{proof}
Write $t=m/n$ with $m\in\Z\setminus\{0\}$, $n\in\Z_{>0}$,
and $\gcd(m,n)=1$. The change of variables $(x,y)=(n^2X,n^3Y)$
is a $\Q$-isomorphism to the nonsingular integral model
\[
 y^2=x^3+4n^2x^2+16m^2n^4,
 \qquad \widetilde P=(0,-4mn^2).
\]
Set $d=2m^2+n^2$. The addition law gives
\[
 [3]\widetilde P=(4m^2,4md),\qquad
 x([6]\widetilde P)=\frac{4(m^6-n^4d)}{d^2},
\]
where $[6]\widetilde P\ne\mathcal O$ because $4md\ne0$.
Since $\gcd(m,d)=1$, we have $\gcd(m^6-n^4d,d)=1$.
Thus integrality of $x([6]\widetilde P)$ would force $d^2\mid4$,
contrary to $d\ge3$. By Nagell--Lutz integrality
\cite[Theorem~2.5, p.~56]{silvermantate2015rational}, every finite rational
torsion point on this integral equation has integral coordinates.
Hence $[6]\widetilde P$, and therefore $P_t$, has infinite order.
The subgroup generated by $P_t$ is isomorphic to $\Z$, proving the rank claim.
\end{proof}

\begin{corollary}\label{cor:infinitely-many-b}
Fix $a\in\Q^\times$. Then there exist infinitely many rational numbers $b$ such that
\[
f_{b,a}(x)=x^3+bx+a
\]
admits a rational $2$-cycle.
\end{corollary}

\begin{proof}
By Theorem~\ref{thm:Pt-infinite}, the points
$[n]P_a=(X_n,Y_n)$ for $n\ge2$ are distinct and finite.
If $X_n=0$, then $[n]P_a$ is one of the two points of $E_a(\Q)$ with $X$-coordinate $0$, namely $\pm P_a$. Thus either $(n-1)P_a=\mathcal O$ or $(n+1)P_a=\mathcal O$, contradicting the infinite order of $P_a$. If $Y_n=0$, then $[n]P_a$ is a rational $2$-torsion point, so $[2n]P_a=\mathcal O$, again a contradiction.

Therefore Theorem~\ref{thm:dictionary} applies to every $[n]P_a$ and yields
\[
b_n=-\frac{X_n}{4}-2-\frac{16a^2}{X_n^2}\in\Q
\]
such that $f_{b_n,a}$ admits a rational $2$-cycle.

For any fixed $b_0\in\Q$, the indices satisfying $b_n=b_0$ have
\[
X_n^3+4(b_0+2)X_n^2+64a^2=0.
\]
This cubic has at most three roots, and each root is the $X$-coordinate
of at most two points on $E_a$. Since the multiples $[n]P_a$ are
distinct, each value $b_0$ occurs for at most six indices. The infinite
sequence of multiples therefore gives infinitely many distinct values
$b_n$, as required.
\end{proof}

\section{The Elliptic Family \texorpdfstring{$\mathcalE_t$}{Et}}\label{sec:elliptic-family}

\subsection{Basic invariants}\label{subsec:family-invariants}

We now determine the torsion term in Theorem~\ref{thm:main-family}.
Under the dictionary of Theorem~\ref{thm:dictionary}, rational points of
\(E_a\) on the excluded divisor \(Y=0\) are exactly its rational points of
order \(2\). The analysis therefore identifies uniformly when this
degenerate divisor has rational points, and then excludes every possible
higher torsion order. The proof uses a Prym quotient at level \(5\) and
two-cover descent with elliptic Chabauty at level \(7\). Since the fiber at
\(t=0\) is singular, we restrict throughout to \(t\in\Q^\times\).

We regard \eqref{eq:Et-family} as a generalized Weierstrass equation with
\[
a_1=a_3=a_4=0,\qquad a_2=4,\qquad a_6=16t^2.
\]
The associated standard invariants are therefore
\begin{equation}\label{eq:Et-bi}
b_2=16,\qquad b_4=0,\qquad b_6=64t^2,\qquad b_8=256t^2,
\end{equation}
and hence
\begin{equation}\label{eq:Et-ci}
c_4=256,\qquad c_6=-4096-13824t^2.
\end{equation}
From \eqref{eq:Et-bi} and \eqref{eq:Et-ci} we obtain
\begin{equation}\label{eq:Et-discriminant}
\Delta(\mathcalE_t)
=-b_2^2b_8-8b_4^3-27b_6^2+9b_2b_4b_6
=-4096t^2(16+27t^2),
\end{equation}
and therefore
\begin{equation}\label{eq:Et-j}
j(\mathcalE_t)=\frac{c_4^3}{\Delta(\mathcalE_t)}
=-\frac{4096}{t^2(16+27t^2)}.
\end{equation}
Because $16+27t^2>0$ for every $t\in\Q$, the discriminant vanishes if and only if $t=0$. Thus \eqref{eq:Et-family} defines an elliptic curve precisely for $t\in\Q^\times$.

For later reference, note also that
\[
j(\mathcalE_t)<0\qquad\text{for every }t\in\Q^\times.
\]
In particular, the rational specializations in this family never have $j$-invariant $0$ or $1728$.

\subsection{Rational 2-torsion}\label{subsec:family-2torsion}

\begin{proposition}\label{prop:Et-2torsion}
Let $t\in\Q^\times$. Then $\mathcalE_t(\Q)$ has a nontrivial rational point of order $2$ if and only if there exists $u\in\Q^\times$ such that
\begin{equation}\label{eq:Et-2torsion-param}
t=\frac{u(u^2+4)}{4}.
\end{equation}
In that case the unique nontrivial rational point of order $2$ is
\begin{equation}\label{eq:Et-2torsion-point}
T_u=\bigl(-(u^2+4),0\bigr),
\end{equation}
and therefore
\[
\mathcalE_t(\Q)[2]\cong \Z/2\Z.
\]
If no such $u$ exists, then $\mathcalE_t(\Q)[2]$ is trivial.
\end{proposition}

\begin{proof}
A rational point on \eqref{eq:Et-family} has order $2$ if and only if it has the form $(r,0)$ with
\begin{equation}\label{eq:Et-2torsion-cubic}
r^3+4r^2+16t^2=0.
\end{equation}
Equivalently,
\begin{equation}\label{eq:Et-2torsion-square}
16t^2=-r^2(r+4).
\end{equation}
Since $t\neq 0$, we also have $r\neq 0$, and \eqref{eq:Et-2torsion-square} gives $-(r+4)=(4t/r)^2$. Write
\[
-(r+4)=u^2,\qquad u\in\Q^\times.
\]
Then
\[
r=-(u^2+4),
\]
and substituting into \eqref{eq:Et-2torsion-square} gives
\[
16t^2=u^2(u^2+4)^2.
\]
Hence
\[
t=\pm \frac{u(u^2+4)}{4}.
\]
Replacing $u$ by $-u$ if necessary yields \eqref{eq:Et-2torsion-param}.

Conversely, if \eqref{eq:Et-2torsion-param} holds, then substituting $X=-(u^2+4)$ into the right-hand side of \eqref{eq:Et-family} gives
\[
X^3+4X^2+16t^2=-(u^2+4)^3+4(u^2+4)^2+u^2(u^2+4)^2=0,
\]
so \eqref{eq:Et-2torsion-point} is indeed a rational point of order $2$.

It remains to prove uniqueness. Writing $\Delta(\mathcalE_t)$ for the
elliptic-curve discriminant in \eqref{eq:Et-discriminant}, the
cubic-polynomial discriminant of $F_t(X):=X^3+4X^2+16t^2$ is
\[
\operatorname{disc}(F_t)=-256t^2(16+27t^2)
=\frac{1}{16}\Delta(\mathcalE_t).
\]
For every $t\in\Q^\times$, this polynomial discriminant is negative.
Therefore the real cubic \(F_t\) has exactly one real root. Hence it has at
most one rational root, and
$\mathcalE_t(\Q)$ has at most one nontrivial rational $2$-torsion point. Thus
the rational $2$-torsion subgroup is either trivial or isomorphic to
$\Z/2\Z$.
\end{proof}

Mazur's theorem \cite{mazur1977modular,mazur1978rational} now reduces the
classification to four torsion orders. The possible groups over $\Q$ are
$\Z/N\Z$ with $1\leq N\leq10$ or $N=12$, and
$\Z/2\Z\times\Z/2N\Z$ with $1\leq N\leq4$.
Proposition~\ref{prop:Et-2torsion} excludes the latter groups. Among the
cyclic groups, orders $6$, $9$, $10$, and $12$ entail a point of order $3$
or $5$, and order $8$ entails a point of order $4$. It therefore suffices
to exclude orders $3$, $5$, $7$, and $4$, as we do in
Sections~\ref{sec:torsion-3}--\ref{sec:torsion-2primary}.

\section{Exclusion of Rational 3-Isogenies}\label{sec:torsion-3}

\begin{theorem}\label{thm:no-3-torsion}
For every $t\in\Q^\times$, the elliptic curve $\mathcalE_t$ admits no rational $3$-isogeny. In particular, $\mathcalE_t(\Q)$ has no point of order $3$.
\end{theorem}

A rational $3$-isogeny would give a rational root of the $3$-division
polynomial. We map such a root to an auxiliary elliptic curve, whose
rational points we determine first.

\begin{proposition}\label{prop:3torsion-auxiliary}
The elliptic curve
\begin{equation}\label{eq:3torsion-Eprime}
E':\quad S^2=R^3-39R-70=(R-7)(R+5)(R+2)
\end{equation}
satisfies
\[
E'(\Q)=\{\mathcal O,\,(-5,0),\,(-2,0),\,(7,0)\}.
\]
\end{proposition}

\begin{proof}
The three affine points are the rational $2$-torsion points visible from the
factorization in \eqref{eq:3torsion-Eprime}. The SageMath certificate
recorded in Appendix~\ref{app:verification} computes the rank with PARI \cite{pari},
with proof mode enabled and database lookup disabled, and returns rank \(0\).
It computes the torsion subgroup as \((\Z/2\Z)^2\), with torsion points
\[
(-5,0),\quad (-2,0),\quad (7,0),
\]
together with the point at infinity. The certified rank-zero result and
torsion computation therefore give the stated group $E'(\Q)$.
\end{proof}

\begin{proof}[Proof of Theorem~\ref{thm:no-3-torsion}]
Let $t\in\Q^\times$. The $3$-division polynomial of
\eqref{eq:Et-family} is
\begin{equation}\label{eq:3torsion-psi3}
\psi_3(x)=3x^4+16x^3+192t^2x+256t^2
=x^3(3x+16)+64t^2(3x+4).
\end{equation}
A rational $3$-isogeny has a Galois-stable kernel
$\{\mathcal O,Q,-Q\}$. Since $Q$ and $-Q$ have the same
$x$-coordinate, this coordinate is rational and is a root of $\psi_3$
\cite[Chapter~III, Exercise~3.7]{silverman2009arithmetic}.
We show that $\psi_3$ has no rational root.

Suppose that $\psi_3(x)=0$ for $x\in\Q$. The values $x=0$ and
$x=-4/3$ do not satisfy this equation, since $t\neq0$. Hence the
rational quantities
\begin{equation}\label{eq:3torsion-RS}
R=\frac{2(14-3x)}{3x+4},\qquad
S=\frac{432t}{x(3x+4)}
\end{equation}
are defined, and $S\neq0$. Direct substitution gives the identity
\[
S^2-(R^3-39R-70)
=\frac{2916\psi_3(x)}{x^2(3x+4)^3}=0,
\]
which is also checked in the SageMath certificate. Thus $(R,S)$ is an
affine rational point of $E'$ with $S\neq0$. This contradicts
Proposition~\ref{prop:3torsion-auxiliary}, whose three affine rational
points all have $S=0$. Therefore $\mathcalE_t$ admits no rational
$3$-isogeny and, in particular, no rational point of order $3$.
\end{proof}

\section{Exclusion of Rational 5-Isogenies}\label{sec:torsion-5}

We pass from the modular fiber product to a smooth plane quartic model
of its normalization, construct the associated Prym, and control its
rank and rational torsion. A fixed-point argument then determines the
rational points of the quartic and excludes rational $5$-isogenies.

A rational \(5\)-isogeny on \(\mathcalE_t\) would define a noncuspidal
rational point of \(X_0(5)\). This modular curve is isomorphic to
\(\mathbb P^1_\Q\), with a Hauptmodul \(\lambda\) for which
\begin{equation}\label{eq:5torsion-j}
j_{0,5}(\lambda)=\frac{(\lambda^2+10\lambda+5)^3}{\lambda},
\end{equation}
and the cusps are \(\lambda=0,\infty\). In terms of
\(h(\tau)=(\eta(\tau)/\eta(5\tau))^6\), this normalization is
\(\lambda=125/h\), and the Fricke involution is
\(\lambda\mapsto125/\lambda\); see
\cite[\S~3.1 and Tables~2 and~7]{maier2009modular}.
Comparing \eqref{eq:5torsion-j} with \eqref{eq:Et-j} gives
\begin{equation}\label{eq:5torsion-eqj}
-\frac{4096}{t^2(16+27t^2)}
=\frac{(\lambda^2+10\lambda+5)^3}{\lambda}.
\end{equation}
Thus any rational \(5\)-isogeny for \(t\neq0\) gives a rational point,
with \(\lambda t\neq0\), on the affine fiber product
\begin{equation}\label{eq:5torsion-affine-fiber}
t^2(16+27t^2)(\lambda^2+10\lambda+5)^3+4096\lambda=0.
\end{equation}
Its normalization has the smooth plane quartic model \(C_5\) below.
The following lemma explains why the Prym rank and torsion bounds
suffice to determine its rational points.

\begin{lemma}\label{lem:prym-fixed-points}
Let \(C/\Q\) be a smooth projective geometrically integral,
non-hyperelliptic curve of genus at least \(3\), with an involution
\(\iota\) defined over \(\Q\) and a rational fixed point \(Q_0\). Put
\[
\mathcal P=\operatorname{im}(1-\iota_*)\subseteq\operatorname{Jac}(C).
\]
If \(2\mathcal P(\Q)=0\), then every rational point of \(C\) is fixed
by \(\iota\).
\end{lemma}

\begin{proof}
For \(Q\in C(\Q)\), the fixed-point hypothesis gives
\[
[Q-\iota(Q)]=(1-\iota_*)[Q-Q_0]\in\mathcal P(\Q).
\]
Hence \(2Q-2\iota(Q)\) is principal. If \(Q\neq\iota(Q)\), a function
with this divisor gives a degree-\(2\) map to \(\mathbb P^1\),
contradicting non-hyperellipticity.
\end{proof}

\begin{proposition}\label{prop:C5-rational-points}
Let \(C_5\) be the smooth plane quartic
\begin{equation}\label{eq:C5-canonical-quartic}
y^4+(6x^2-8xw+w^2)y^2+x(x+w)(9x^2+3xw-w^2)=0.
\end{equation}
Then
\[
C_5(\Q)=\{(0:0:1),\,(-1:0:1)\}.
\]
\end{proposition}

\begin{proof}
Let \(\iota(x:y:w)=(x:-y:w)\), let
\(\pi:C_5\to D=C_5/\langle\iota\rangle\) be the quotient map, and put
\(J=\operatorname{Jac}(C_5)\). The certificates in
Appendix~\ref{app:verification} verify that \(C_5\) is smooth and
non-hyperelliptic of genus \(3\), and that \(D\) has elliptic model
\[
E:\quad v^2=u^3-u^2+12u+72.
\]
The Prym variety of this double cover is
\[
\mathcal P=\ker\!\bigl(\operatorname{Nm}_\pi:J\to
\operatorname{Jac}(D)\bigr)^0
=\operatorname{im}(1-\iota_*).
\]
Indeed, this image is connected, lies in the norm kernel, and has
dimension \(g(C_5)-g(D)=2\). Since \(Q_0=(0:0:1)\) is a rational
fixed point, Lemma~\ref{lem:prym-fixed-points} reduces the rational-point
problem to proving \(2\mathcal P(\Q)=0\).

To identify \(\mathcal P\), write \eqref{eq:C5-canonical-quartic} as
\(y^4-hy^2+fg=0\), where
\[
f=x(x+w),\qquad h=-6x^2+8xw-w^2,\qquad
 g=9x^2+3xw-w^2.
\]
The coefficient matrix in the basis \(x^2,xw,w^2\) is
\[
A=\begin{pmatrix}1&1&0\\-6&8&-1\\9&3&-1\end{pmatrix},
\qquad \det A=-20.
\]
These splitting and nondegeneracy data allow us to apply the explicit
Prym construction of Ritzenthaler and Romagny
\cite[Theorem~1.1]{ritzenthaler2018prym}. It gives a genus-\(2\) curve
whose Jacobian is \(\Q\)-isogenous to \(\mathcal P\), with equation
\[
z^2=-\frac{(6s^2-2s+1)(324s^4+144s^3+14s^2-6s-1)}{2000}.
\]
The substitution \(v=100z\) gives
\begin{equation}\label{eq:P5-prym-curve}
P_5:\quad
v^2=-5(6s^2-2s+1)(324s^4+144s^3+14s^2-6s-1).
\end{equation}
Write \(J(P_5)=\operatorname{Jac}(P_5)\). Rank is invariant under the
\(\Q\)-isogeny \(\mathcal P\sim J(P_5)\).

The genus-\(2\) descent \cite{stoll2001descent} computes
\[
\dim_{\mathbf F_2}\operatorname{Sel}^{(2)}(J(P_5)/\Q)=2,
\qquad \dim_{\mathbf F_2}J(P_5)(\Q)[2]=1,
\]
so the initial rank upper bound is \(1\). For this genus-\(2\) curve,
a place is deficient if no divisor of degree \(1\) is defined over the
corresponding completion of \(\Q\). The unique deficient place of
\(P_5\) is \(p=2\). An odd number of deficient places improves this
upper bound by \(1\), without assuming finiteness of the
Tate--Shafarevich group
\cite[Theorem~5 and Corollary~12]{poonenStollCasselsTate}; see explicitly
\cite[Section~3, Lemma~3.1]{fisherYanCasselsTate}. Hence
\(\rank J(P_5)(\Q)=\rank\mathcal P(\Q)=0\), in agreement with the
certified output \(\operatorname{RankBounds}(J(P_5))=(0,0)\).
Appendix~\ref{app:verification} records the corroborating explicit
Cassels--Tate pairing calculation.

The torsion bound is separate. The discriminant of the sextic defining
\(P_5\) in \eqref{eq:P5-prym-curve} is supported only at \(2,3,5\), so \(J(P_5)\) has good reduction at \(7,11,13\).
Good reduction is invariant under isogeny, and the \(\Q\)-isogeny
\(\mathcal P\sim J(P_5)\) induces isomorphic rational \(\ell\)-adic
Tate modules. Thus both varieties have the same Frobenius polynomials
and reduction orders at these primes
\cite[Chapter~IV, Theorem~3.5, Corollary~3.6 and the following
discussion]{milne2008abelian}. The certificate gives
\[
\#J(P_5)(\mathbb F_7)=44,\qquad
\#J(P_5)(\mathbb F_{11})=108,\qquad
\#J(P_5)(\mathbb F_{13})=134.
\]
Prime-to-residue-characteristic torsion injects under good reduction
\cite[Chapter~IV, proof of Theorem~3.5]{milne2008abelian}.
For \(\ell\notin\{7,11,13\}\), the order of the \(\ell\)-primary
rational torsion of \(\mathcal P\) divides \(\gcd(44,108,134)=2\).
For \(\ell=7,11,13\), the other two good primes give greatest common
divisors \(2,2,4\), respectively, so these three primary parts are
trivial. Therefore \(\#\mathcal P(\Q)_{\rm tors}\leq2\).
Together with rank zero, this proves \(2\mathcal P(\Q)=0\).

Lemma~\ref{lem:prym-fixed-points} now forces every rational point of
\(C_5\) into the fixed locus of \(\iota\). The only projective fixed
point with \(y\neq0\) is \((0:1:0)\), which is not on \(C_5\).
Setting \(y=0\) gives
\[
x(x+w)(9x^2+3xw-w^2)=0.
\]
There are no such points with \(w=0\), and the last quadratic has
discriminant \(45\) on \(w=1\), a nonsquare in \(\Q\). Thus the rational
fixed points are exactly \((0:0:1)\) and \((-1:0:1)\).
\end{proof}

\begin{theorem}\label{thm:no-5-torsion}
For every \(t\in\Q^\times\), the elliptic curve \(\mathcalE_t\) admits
no rational \(5\)-isogeny. In particular, \(\mathcalE_t(\Q)\) has no
point of order \(5\).
\end{theorem}

\begin{proof}
Suppose that \(\mathcalE_t\) admits a rational \(5\)-isogeny for some
\(t\in\Q^\times\). Its Galois-stable cyclic kernel defines a
noncuspidal rational point of \(X_0(5)\), so the opening modular
calculation gives a rational point \((\lambda,t)\) on
\eqref{eq:5torsion-affine-fiber} with \(\lambda t\neq0\).
Let \(\mathcal C_5\) be the projective closure of this fiber product,
in coordinates \((L:T:Z)\), where \(\lambda=L/Z\) and \(t=T/Z\):
\begin{equation}\label{eq:5torsion-projective-fiber}
T^2(16Z^2+27T^2)(L^2+10LZ+5Z^2)^3+4096LZ^9=0.
\end{equation}
The certificate verifies that the normalization of \(\mathcal C_5\)
has genus \(3\) and that the reduced geometric singular scheme of
\(\mathcal C_5\) has degree \(6\). Two singularities are \((0:1:0)\) and \((1:0:0)\); the other
four form the affine degree-\(4\) scheme
\[
\lambda^2+4\lambda-1=0,\qquad t^2+\frac{8}{27}=0.
\]
Thus every rational affine point is smooth. The smooth canonical model
is \(C_5\) in Proposition~\ref{prop:C5-rational-points}; the certificate
passes from Magma's canonical coordinates \mbox{\((X:Y:W)\)} to
\eqref{eq:C5-canonical-quartic} by
\[
X=\frac{64}{3}x,\qquad Y=16y,\qquad W=w.
\]
In these smaller coordinates, write the resulting canonical map as
\[
\phi=(\Phi_1:\Phi_2:\Phi_3):\mathcal C_5\dashrightarrow C_5.
\]
On \(Z=1\), its second coordinate is
\[
\Phi_2(\lambda,t,1)
=\frac{729}{1024}\lambda\left(t^2+\frac{8}{27}\right)
(\lambda^2+10\lambda+5).
\]
Every factor is nonzero for rational \(\lambda,t\) with
\(\lambda t\neq0\): \(t^2+8/27\) has no rational root, and
\(\lambda^2+10\lambda+5\) has nonsquare discriminant \(80\).
Hence the canonical coordinate tuple is defined there and gives a
rational point of \(C_5\) with \(y\neq0\), contradicting
Proposition~\ref{prop:C5-rational-points}. This excludes every rational
\(5\)-isogeny. A rational point of order \(5\) would generate a
Galois-stable cyclic kernel, so it too is excluded.
\end{proof}

\begin{remark}\label{rem:5-method}
The degree-\(6\) \(j\)-map of \(X_0(5)\) detects rational cyclic
subgroups and gives the stronger no-isogeny conclusion above. At level
\(7\), a rational torsion point supplies a Tate normal form whose
\(c_4\)-invariant must be a rational fourth power. This condition leads
to the hyperelliptic genus-\(3\) curve treated by two-cover descent and
elliptic Chabauty in the next section. These different modular inputs
account for the different Diophantine methods.
\end{remark}

\section{Exclusion of Rational 7-Torsion}\label{sec:torsion-7}

A rational point of order $7$ would force the $c_4$-invariant of its
Tate normal form to be a rational fourth power. We first determine the
rational points on the genus-$3$ curve $C_7$ defined by the weaker
square condition.
Two-cover descent gives six possible descent classes, and degree-$4$
norm maps associate them with six genus-$1$ quartics over a biquadratic
field. Elliptic Chabauty determines their rational coordinate images.
We then eliminate the noncuspidal parameters among these points using
the stronger fourth-power condition.

An elliptic curve over $\Q$ with a rational point $P$ of order $7$ is
$\Q$-isomorphic to a Tate normal form
\[
E_\mu:\quad y^2+(1-c)xy-by=x^3-bx^2,\qquad P=(0,0),
\]
where
\[
b=\mu^2(\mu-1),\qquad c=\mu(\mu-1),\qquad \mu\in\Q;
\]
see the entry $N=7$ of Kubert's table \cite[Table~1]{kubert1976universal}.
The discriminant is $\mu^7(\mu-1)^7D_7(\mu)$, where
$D_7(\mu)=\mu^3-8\mu^2+5\mu+1$. Thus $\mu=0,1$, the roots of $D_7$,
and $\mu=\infty$ are cusps and do not give nonsingular elliptic curves.
Set
\[
A_7(\mu)=\mu^2-\mu+1,\qquad
B_7(\mu)=\mu^6-11\mu^5+30\mu^4-15\mu^3-10\mu^2+5\mu+1.
\]

\begin{lemma}\label{lem:7torsion-fourth-power}
If $\mathcalE_t$ has a rational point of order $7$ for some
$t\in\Q^\times$, then there is a noncuspidal $\mu\in\Q$ such that
\begin{equation}\label{eq:7torsion-fourth-power}
A_7(\mu)B_7(\mu)\in\Q^{\times4}.
\end{equation}
In particular, this parameter gives a rational point on
\begin{equation}\label{eq:7torsion-C7}
C_7:\quad y^2=A_7(\mu)B_7(\mu).
\end{equation}
\end{lemma}

\begin{proof}
For the displayed Tate normal form,
\[
b_2=(1-c)^2-4b,\qquad b_4=-b(1-c),
\]
so direct substitution in $c_4=b_2^2-24b_4$ gives
\[
c_4(E_\mu)=A_7(\mu)B_7(\mu).
\]
On the other hand, $c_4(\mathcalE_t)=256=4^4$ by
\eqref{eq:Et-ci}. Since $E_\mu$ and $\mathcalE_t$ are
$\Q$-isomorphic, their nonzero $c_4$-invariants differ by the fourth
power of a rational Weierstrass scaling factor. Hence
$A_7(\mu)B_7(\mu)=v^4$ for some $v\in\Q^\times$, and
$(\mu,v^2)\in C_7(\Q)$. The invariant identities are also checked in
the Magma scripts in the online computational supplement.
\end{proof}

Magma verifies \(\gcd(A_7B_7,(A_7B_7)')=1\). Thus \(A_7(\mu)B_7(\mu)\) is a squarefree polynomial of degree \(8\), and the standard genus formula for an even-degree hyperelliptic curve \(y^2=f(\mu)\) gives genus \((8-2)/2=3\) for the smooth projective model of \(C_7\). Moreover \(A_7\) has discriminant \(-3\), and \(B_7\) has no rational root by the rational-root theorem since \(B_7(1)=1\) and \(B_7(-1)=43\). Hence \(A_7B_7\) has no rational root, so \(C_7\) has no rational finite Weierstrass point.

\begin{lemma}\label{lem:C7-cover-bridge}
Let
\[
f(\mu)=A_7(\mu)B_7(\mu),\qquad \mathscr A=\Q[\theta]/(f(\theta)).
\]
Put
\[
M=\Q(s,r),\qquad s^2=21,\quad r^2=-3,
\]
so that $[M:\Q]=4$, and set
\[
\alpha=\frac{1+r}{2},\qquad
B_+(X)=X^3-\frac{11+s}{2}X^2+\frac{5+s}{2}X+1,
\]
\[
g(X)=(X-\alpha)B_+(X).
\]
Then $g$ is a degree-$4$ factor of $f$ over $M$. Let
\[
L=M[\Theta]/(g(\Theta))
\]
be the degree-$4$ finite \'{e}tale $M$-algebra, and let
\[
j:\mathscr A\otimes_\Q M\longrightarrow L,\quad \theta\longmapsto\Theta,
\]
and, for a chosen lift $\delta\in\mathscr A^\times$ of a fake-$2$-Selmer class, put
\[
\gamma_\delta=N_{L/M}(j(\delta)).
\]
If an affine non-Weierstrass point $P=(\mu,y)\in C_7(\Q)$ has fake-Selmer class represented by $\delta$, then there exists an $M$-rational point on
\begin{equation}\label{eq:C7-elliptic-cover}
H_\delta:\quad Y^2=\gamma_\delta g(X)
\end{equation}
whose $X$-coordinate is $\mu$. In particular, the rational-scalar ambiguity in the fake-Selmer class introduces no additional twists in \eqref{eq:C7-elliptic-cover}.
\end{lemma}

\begin{proof}
Over $M$ one has
\[
f(X)=(X-\alpha)(X-\bar\alpha)B_+(X)B_-(X),
\]
where $\bar\alpha=(1-r)/2$ and
\[
B_-(X)=X^3-\frac{11-s}{2}X^2+\frac{5-s}{2}X+1;
\]
hence $g$ has degree $4$ and divides $f$. For the monic even-degree model $y^2=f(\mu)$, Bruin--Stoll descent associates to an affine non-Weierstrass point the class of $\mu-\theta$ in $\mathscr A^\times/\Q^\times\mathscr A^{\times2}$ \cite[Section~2]{bruin2009two}. Thus there are $c\in\Q^\times$ and $\xi\in(\mathscr A\otimes_\Q M)^\times$ such that
\[
\mu-\theta=c\,\delta\,\xi^2.
\]
Applying $j$ and taking the degree-$4$ norm gives
\[
g(\mu)=c^4\gamma_\delta N_{L/M}(j(\xi))^2.
\]
Consequently
\[
\left(c^2\gamma_\delta N_{L/M}(j(\xi))\right)^2
=\gamma_\delta g(\mu),
\]
which is the required point on $H_\delta$. The exponent $4$ is essential: it makes the norm of every rational scalar a square. More generally, replacing $\delta$ by another lift $q\delta\eta^2$, with $q\in\Q^\times$, multiplies $\gamma_\delta$ by the square $q^4N_{L/M}(j(\eta))^2$. Thus the square class of the quartic twist is independent of the chosen lift, and no rational scalar twists are omitted.
\end{proof}

\begin{proposition}\label{prop:C7-rational-points}
The rational points on $C_7$ are exactly the two points at infinity together with
\[
(0,\pm1),\quad (1,\pm1),\quad
\left(-\frac12,\pm\frac7{16}\right),\quad
\left(\frac23,\pm\frac7{81}\right),\quad
(3,\pm7).
\]
\end{proposition}

\begin{proof}
The two points at infinity are rational because the defining polynomial is
monic of even degree. Let \(P\) be an affine rational point. Since \(f\) has
no rational root, \(P\) is non-Weierstrass; hence
Lemma~\ref{lem:C7-cover-bridge} applies.

We use the Bruin--Stoll two-cover descent algorithm \cite[Sections~2--3]{bruin2009two} as a finite covering step and Bruin's elliptic Chabauty method \cite[Sections~3--4]{bruin2003chabauty} to determine the rational images on the resulting genus-$1$ quartics. The Magma routines used for the fake-Selmer and elliptic-Chabauty stages are documented in the Magma Handbook \cite{magmaHandbookTwoSelmer,magmaHandbookChabauty}. Magma computes the fake $2$-Selmer set of $C_7$ in
\[
\mathscr A^\times/\Q^\times\mathscr A^{\times2}
\]
and verifies that its six classes are represented by the following chosen lifts:
\[
\delta\in
\left\{
1,\,-\theta,\,1-\theta,\,-\frac12-\theta,\,\frac23-\theta,\,3-\theta
\right\}
\subset\mathscr A^\times,
\qquad
\mathscr A=\Q[\theta]/(f(\theta)),
\]
where
\[
f(\mu)=(\mu^2-\mu+1)(\mu^6-11\mu^5+30\mu^4-15\mu^3-10\mu^2+5\mu+1).
\]
Thus every affine point of $C_7(\Q)$ lifts to one of the corresponding two-covers.
Here ``represented by $\delta_i$'' is meant in the Bruin--Stoll formalism: the displayed elements are chosen lifts of the complete locally soluble covering set. Lemma~\ref{lem:C7-cover-bridge}, rather than a normalization convention for these lifts, is what removes the rational-scalar ambiguity.

It remains to determine the rational images from the six fixed quartics $H_i/M$ of \eqref{eq:C7-elliptic-cover}. With $s^2=21$ and $r^2=-3$, their norm factors are
\[
\begin{array}{c|c|c}
i&\delta_i&\gamma_i\\ \hline
1&1&1\\
2&-\theta&-(1+r)/2\\
3&1-\theta&(r-1)/2\\
4&-\tfrac12-\theta&(2+r)(14+3s)/16\\
5&\tfrac23-\theta&(1-3r)(14+3s)/162\\
6&3-\theta&-(5-r)(14+3s)/2.
\end{array}
\]
For $i=1$ the monic quartic $H_1$ has an $M$-rational point at infinity. For $i\geq2$, if $x_i=0,1,-\tfrac12,\tfrac23,3$, respectively, then $\gamma_i=g(x_i)$, so $(x_i,\gamma_i)\in H_i(M)$. These base points identify the $H_i$ with elliptic curves $E_i/M$.
Let $\pi_i:E_i\to\mathbb P^1_M$ denote the transported $X$-coordinate
map, which has degree $2$.

For every $i$, Magma's \texttt{PseudoMordellWeilGroup} returns
\texttt{success = true} and invariants $[0,0]$. Thus the supplied subgroup
has finite odd index in $E_i(M)$ \cite[Example~H133E24]{magmaHandbookChabauty} and
\[
\operatorname{rank}E_i(M)=2<4=[M:\Q],
\]
as required for elliptic Chabauty. At each of the two primes of \(M\)
above \(5\), the certificate verifies good reduction of the chosen
elliptic model and of the degree-\(2\) map \(\pi_i\). It checks smooth
integral quartic models, the exact transport of their \(X\)-maps to
\(E_i\), and preservation of degree under reduction, including the
fibers above infinity; the local data and justification are recorded in
the online computational supplement. After saturating the two free generator images through
$100$, we apply the single-prime routine \texttt{Chabauty} at $p=5$.
Each cover returns
\[
N_i=6,\qquad \#V_i=6,\qquad R_i=1.
\]
The documented completeness criterion states that $N_i$ bounds the rational-image points on the full group whenever the index of the supplied subgroup is finite and coprime to $R_i$ \cite{magmaHandbookChabauty}. The two saturated generators remain independent, so the subgroup passed to \texttt{Chabauty} has rank $2$ and therefore finite index in $E_i(M)$. Since $R_i=1$, coprimality is automatic; equality $N_i=\#V_i$ therefore proves that the rational-image points are exhaustive. For each $i$, the six points counted by $\#V_i$ lie on $E_i(M)$ and have three distinct images under $\pi_i$. The six fresh-process computations give
\[
\begin{array}{c|c|c|c|c}
i&N_i&\#V_i&R_i&\pi_i(E_i(M))\cap\mathbb P^1(\Q)\\ \hline
1,2,3&6&6&1&\{\infty,0,1\}\\
4,5,6&6&6&1&\{-\tfrac12,\tfrac23,3\}.
\end{array}
\]
Hence the union of the certified rational images is
\[
\mu\in\left\{\infty,\,0,\,1,\,-\frac12,\,\frac23,\,3\right\}.
\]
Pulling these six values back to $C_7$ gives exactly the two points at infinity and the ten affine points listed in the proposition. The exact field model, norm factors, seed, rank condition, Chabauty outputs, rational-image assertions, and pullback assertions are recorded in Appendix~\ref{app:verification}.
\end{proof}

\begin{remark}\label{rem:C7-order-three}
The two triples of coordinate values in the table are orbits of an
order-$3$ symmetry. Writing $f=A_7B_7$, the exact identity
\[
(1-\mu)^8f\!\left(\frac1{1-\mu}\right)=f(\mu)
\]
shows that
\[
(\mu,y)\longmapsto
\left(\frac1{1-\mu},\frac{y}{(1-\mu)^4}\right)
\]
extends to an order-$3$ automorphism of the smooth projective curve
$C_7$. Its orbits on the displayed coordinate values are
\[
\infty\longmapsto0\longmapsto1\longmapsto\infty,
\qquad
-\frac12\longmapsto\frac23\longmapsto3\longmapsto-\frac12.
\]
This symmetry explains the grouping of the images; the completeness
proof above retains all six elliptic Chabauty computations.
\end{remark}

\begin{theorem}\label{thm:no-7-torsion}
For every $t\in\Q^\times$, the elliptic curve $\mathcalE_t$ has no rational point of order~$7$.
\end{theorem}

\begin{proof}
By Lemma~\ref{lem:7torsion-fourth-power} and
Proposition~\ref{prop:C7-rational-points}, a rational point of order $7$
would give a parameter $\mu\in\{-\tfrac12,\tfrac23,3\}$, since
$\infty,0,1$ are cusps. At these three values, respectively,
\[
A_7(\mu)B_7(\mu)=\frac{49}{256},\quad\frac{49}{6561},\quad49.
\]
Each has $7$-adic valuation $2$, whereas the $7$-adic valuation of a
nonzero rational fourth power is divisible by $4$. This contradicts
\eqref{eq:7torsion-fourth-power}.
\end{proof}

\begin{corollary}\label{cor:remaining-torsion}
Let \(t\in\Q^\times\). Then \(\mathcalE_t(\Q)_{\mathrm{tors}}\) is either trivial or cyclic of order \(2\), \(4\), or \(8\). It is nontrivial only if \(t\) lies in the locus \eqref{eq:Et-2torsion-param}, in which case its unique element of order \(2\) is the point \(T_u\) of Proposition~\ref{prop:Et-2torsion}.
\end{corollary}

\begin{proof}
By Theorems~\ref{thm:no-3-torsion},~\ref{thm:no-5-torsion}, and~\ref{thm:no-7-torsion}, the group \(\mathcalE_t(\Q)_{\mathrm{tors}}\) contains no element of order \(3\), \(5\), or \(7\), hence none of order \(6\), \(9\), \(10\), or \(12\). Among the groups in Mazur's classification \cite{mazur1977modular,mazur1978rational}, this leaves the trivial group and \(\Z/2\Z\), \(\Z/4\Z\), and \(\Z/8\Z\); the groups \(\Z/2\Z\times\Z/2N\Z\) are excluded by Proposition~\ref{prop:Et-2torsion}, which also shows that a point of order \(2\) exists only for \(t\) in the locus \eqref{eq:Et-2torsion-param} and is then the unique point \(T_u\).\hspace*{1em}
\end{proof}

\section{Exclusion of Rational 4-Torsion and the Torsion Classification}\label{sec:torsion-2primary}

By Corollary~\ref{cor:remaining-torsion}, the only possible nontrivial torsion in the family is cyclic of order $2$, $4$, or $8$, supported on the locus of Proposition~\ref{prop:Et-2torsion}. The halving criterion leads to a genus-one quartic. A $2$-isogeny reduces its rank calculation to a full $2$-descent, where a local obstruction at $3$ excludes the only remaining coset of square classes. This excludes orders $4$ and $8$ and completes the torsion classification. The detailed arithmetic proof is given in Appendix~\ref{app:C4-descent}; the SageMath script recorded in Appendix~\ref{app:verification} serves as an exact corroborating calculation.

\begin{lemma}[Halving criterion]\label{lem:halving}
Let $E:\;y^2=x\,(x^2+Ax+B)$ with $A,B\in\Q$ and $B(A^2-4B)\neq0$, and let $T=(0,0)$. Then $T\in2E(\Q)$ if and only if $B=\beta^2$ for some $\beta\in\Q^\times$ and at least one of $A+2\beta$, $A-2\beta$ is a square in $\Q$.
\end{lemma}

\begin{proof}
On this model the duplication formula reads
\[
x([2]P)=\frac{(x^2-B)^2}{4y^2},\qquad P=(x,y),\ y\neq0.
\]
Since $T$ is the unique point of $E$ with $x$-coordinate $0$, a point $P=(x,y)$ with $y\neq0$ satisfies $[2]P=T$ if and only if $x^2=B$.

Suppose such a $P$ exists with $x,y\in\Q$, and put $\beta=x$, so $B=\beta^2$ with $\beta\in\Q^\times$. Then
\[
y^2=\beta\,(\beta^2+A\beta+B)=\beta^2(A+2\beta),
\]
so $A+2\beta$ is a square. Conversely, suppose $B=\beta^2$ and, after replacing $\beta$ by $-\beta$ if necessary (which interchanges $A+2\beta$ and $A-2\beta$), that $A+2\beta=\gamma^2$ with $\gamma\in\Q$. Then $\gamma\neq0$, since $A=-2\beta$ would give $A^2=4B$. The point $Q=(\beta,\beta\gamma)$ lies on $E$, has $y(Q)\neq0$, and satisfies $x(Q)^2=B$, so $[2]Q=T$ and $Q$ has order~$4$.\hspace*{1em}
\end{proof}

\begin{proposition}\label{prop:4torsion-reduction}
Let $t\in\Q^\times$ and suppose that $\mathcalE_t(\Q)$ contains a point of order $4$. Then $t=u(u^2+4)/4$ for some $u\in\Q^\times$, and
\begin{equation}\label{eq:4torsion-B-square}
(u^2+4)(3u^2+4)\in\Q^{\times2}.
\end{equation}
\end{proposition}

\begin{proof}
The double of a point of order $4$ is a rational point of order $2$, so Proposition~\ref{prop:Et-2torsion} gives $t=u(u^2+4)/4$ with $u\in\Q^\times$, and the unique rational point of order $2$ is $T_u=(-(u^2+4),0)$. The translation $X=W-(u^2+4)$ takes \eqref{eq:Et-family} to
\begin{equation}\label{eq:4torsion-shifted}
y^2=W\bigl(W^2-(3u^2+8)W+(u^2+4)(3u^2+4)\bigr),
\end{equation}
with $T_u$ at $(0,0)$: indeed $X^3+4X^2+16t^2=(X+u^2+4)\bigl(X^2-u^2X+u^2(u^2+4)\bigr)$ for this value of $t$, and shifting the first factor to $W$ produces \eqref{eq:4torsion-shifted}. Here
\[
B=(u^2+4)(3u^2+4)\neq0,
\qquad A^2-4B=-u^2(3u^2+16)\neq0,
\]
so Lemma~\ref{lem:halving} applies. A rational point of order $4$ forces
$T_u\in2\mathcalE_t(\Q)$, and the lemma requires $B$ to be a rational square.
\end{proof}

\begin{remark}\label{rem:4torsion-second-condition}
Lemma~\ref{lem:halving} imposes a second condition, that $A+2\beta$ or $A-2\beta$ be a square, where here $A=-(3u^2+8)$. If $B$ is a square, choose $\beta>0$ with $\beta^2=B$. Since $4B-(3u^2+8)^2=3u^4+16u^2>0$, one has $A+2\beta>0$ and $A-2\beta<0$, so only $A+2\beta$ can be a square. We will not need this: the necessary condition \eqref{eq:4torsion-B-square} already fails for every $u\in\Q^\times$.
\end{remark}

\begin{theorem}\label{thm:C4-points}
The smooth projective curve over $\Q$ with affine model
\[
C_4:\quad v^2=3u^4+16u^2+16=(u^2+4)(3u^2+4)
\]
has exactly two rational points, namely $(u,v)=(0,\pm4)$.
\end{theorem}

\begin{proof}
The maps
\begin{equation}\label{eq:C4-to-J4}
x=\frac{8(v+4)}{u^2},\qquad y=\frac{8(x+16)}{u},
\end{equation}
with inverse $u=8(x+16)/y$, $v=xu^2/8-4$, identify $C_4$ with
\[
J_4:\quad y^2=(x+16)(x^2-192).
\]
Proposition~\ref{prop:J4-arithmetic} in Appendix~\ref{app:C4-descent}
verifies this isomorphism, including the exceptional points, and proves
\[
J_4(\Q)=\{\mathcal O,(-16,0)\}.
\]
The arithmetic proof uses a $2$-isogeny to $z^2=x(x+1)(x+3)$ and a
complete $2$-descent, followed by a good-reduction torsion bound.
The two points displayed above correspond to $(0,4)$ and $(0,-4)$ on
$C_4$, respectively, proving the assertion.
\end{proof}

\begin{theorem}\label{thm:no-4-torsion}
For every $t\in\Q^\times$, the elliptic curve $\mathcalE_t$ has no rational point of order $4$, and hence none of order $8$.
\end{theorem}

\begin{proof}
A rational point of order $4$ would produce, by Proposition~\ref{prop:4torsion-reduction}, an element $u\in\Q^\times$ with $(u^2+4)(3u^2+4)=v^2$ for some $v\in\Q$, that is, a rational point of $C_4$ with $u\neq0$. By Theorem~\ref{thm:C4-points} no such point exists. A rational point of order $8$ would have a multiple of order $4$.
\end{proof}

\begin{theorem}[Torsion classification]\label{thm:torsion-classification}
Let $t\in\Q^\times$. Then
\[
\mathcalE_t(\Q)_{\mathrm{tors}}\cong
\begin{cases}
\Z/2\Z, & \text{if } t=\dfrac{u(u^2+4)}{4} \text{ for some } u\in\Q^\times,\\[1.5ex]
0, & \text{otherwise,}
\end{cases}
\]
and in the first case $u$ is unique and the nontrivial torsion point is $T_u=\bigl(-(u^2+4),0\bigr)$.
\end{theorem}

\begin{proof}
By Corollary~\ref{cor:remaining-torsion}, $\mathcalE_t(\Q)_{\mathrm{tors}}$ is trivial or cyclic of order $2$, $4$, or $8$; it is nontrivial only for $t$ in the displayed locus, and its element of order $2$ is then $T_u$. Theorem~\ref{thm:no-4-torsion} excludes the orders $4$ and $8$. The uniqueness of $u$ follows since $u\mapsto u(u^2+4)/4$ has derivative $(3u^2+4)/4>0$ and is therefore injective on $\mathbb{R}$.
\end{proof}

\begin{example}\label{ex:torsion-classification}
For $u=1$ one has $t=5/4$, and the fiber $\mathcalE_{5/4}:Y^2=X^3+4X^2+25$ satisfies $\mathcalE_{5/4}(\Q)_{\mathrm{tors}}=\{\mathcal O,\,(-5,0)\}$ by Theorem~\ref{thm:torsion-classification}, while $P_{5/4}=(0,-5)$ has infinite order by Theorem~\ref{thm:Pt-infinite}. Hence $\mathcalE_{5/4}(\Q)\cong\Z^r\times\Z/2\Z$ with $r\geq1$.
\end{example}

\section*{Data Availability}
All scripts, raw outputs, software-version information, manifests, checksums,
and licensed third-party files needed to reproduce the computer-assisted
arguments are available in the public GitHub repository at
\url{https://github.com/chatchawanpan-dev/rational-2-cycles-cubic-computations}.
The snapshot corresponding to this manuscript is the directory
\path{computational_files/} at commit
\href{https://github.com/chatchawanpan-dev/rational-2-cycles-cubic-computations/tree/a264cde0ba01013c3d5d0a7af3252b3294e2b37b/computational_files}{\texttt{a264cde0ba01}}.
No external datasets were used.

\section*{Competing Interests}
The authors declare no competing interests.

\section*{Funding}
This work was supported by the Fundamental Fund of Khon Kaen University,
which has received funding support from the National Science, Research and
Innovation Fund (NSRF) of Thailand.

\section*{Acknowledgments}
The authors thank Teerapol Sukhonwimolmal for helpful discussions during the
development of this project.

\appendix

\section{Computational Appendix}\label{app:computational}

\subsection{Symbolic identities}\label{app:divpoly}
The SageMath identity script verifies the tripling formula and the direct
expression for $x([6]\widetilde P)$ used in Section~\ref{sec:rank}, together
with the numerator congruence behind the coprimality argument. It also checks
the doubling and tangent-slope formulas. These exact identities corroborate
the self-contained Nagell--Lutz proof.

For the equation $y^2=x^3+a_2x^2+a_4x+a_6$, the
$3$-division polynomial is
\[
\psi_3=3x^4+4a_2x^3+6a_4x^2+12a_6x+4a_2a_6-a_4^2;
\]
see \cite[Chapter~III, Exercise~3.7]{silverman2009arithmetic} and
\cite[Section~3.2]{washington2008elliptic}. Substitution of
$(a_2,a_4,a_6)=(4,0,16t^2)$ gives the polynomial used in
Section~\ref{sec:torsion-3}. The online computational supplement also retains seven exact
division-polynomial evaluations at $P_t$, $[2]P_t$, and $[4]P_t$
on the short Weierstrass model obtained by $X=\xi-4/3$;
these additional checks are not inputs to the rank proof.

\subsection{Proof certificates and reproducibility}\label{app:verification}

The directory \path{computational_files/} in the GitHub snapshot identified
in Data Availability contains exact
scripts and raw outputs for SageMath~10.9 \cite{sagemath} and Magma~V2.29-6
\cite{magma}. The runner rejects error and interruption diagnostics and requires
successful termination and a final completion marker from every job. Instructions, versions, fixed seeds, inventory, and hashes are in
\path{README.md}, \path{SOFTWARE_VERSIONS.txt}, \path{MANIFEST.txt}, and
\path{SHA256SUMS.txt}. In Table~\ref{tab:proof-audit}, S denotes the SageMath identity and
auxiliary elliptic-rank script, P the Prym rank, torsion and fixed-locus
script, C the explicit Cassels--Tate pairing script, and E the six
elliptic-Chabauty jobs. Their full paths and commands are recorded in
\path{PROOF_AUDIT.md}.

The normalization certificate exports the full canonical-map tuple and
all linear model changes to \path{c5_canonical_map_checkpoint.txt} in
\path{magma_5torsion_verification/results/}. The six files
\path{c7_cover_1.txt} through \path{c7_cover_6.txt} in
\path{magma_7torsion_verification/elliptic_chabauty_route/checkpoints/}
record the exact fields, elliptic models, generators, and maps passed to
Chabauty. Their accompanying guide specifies the coordinate conventions.
The certificates generate these static records from the same objects
used in the verified calculations.

\begin{table}[htbp]
\caption{Proof-critical rank and completeness certificates.}
\label{tab:proof-audit}
\centering
\footnotesize
\setlength{\tabcolsep}{3pt}
\renewcommand{\arraystretch}{1.16}
\begin{tabular}{@{}>{\raggedright\arraybackslash}p{0.16\textwidth}>{\raggedright\arraybackslash}p{0.29\textwidth}>{\raggedright\arraybackslash}p{0.47\textwidth}@{}}
\hline
Result / script & Command and exact output & Hypotheses and proof interpretation \\
\hline
Proposition~\ref{prop:3torsion-auxiliary}; S
& \texttt{rank}: \(0\); torsion invariants \((2,2)\)
& PARI, proof mode enabled, database disabled. The rank call proves its
answer or raises an exception \cite{sageRankDocumentation}; hence the four
torsion points exhaust \(E'(\Q)\). \\[3pt]
Proposition~\ref{prop:C5-rational-points}; P
& \texttt{RankBounds}: \((0,0)\); unique deficient place \(2\)
& Selmer and rational \(2\)-torsion dimensions are \(2,1\). The
deficient-place obstruction lowers the rank bound without assuming finite
Sha \cite[Theorem~5 and Corollary~12]{poonenStollCasselsTate};
\cite{magmaHandbookJacobian}. \\[3pt]
Proposition~\ref{prop:C5-rational-points}; C
& Cassels--Tate matrix rank \(1\); radical dimension \(1\)
& Selmer dimension \(2\), torsion dimension \(1\); the rational Kummer
image lies in the radical, giving rank \(0\)
\cite[Section~3]{fisherYanCasselsTate}. \\[3pt]
Proposition~\ref{prop:C7-rational-points}; E
& \texttt{TwoCoverDescent}: six classes
& The displayed representatives equal the complete fake Selmer set
\cite{magmaHandbookTwoSelmer}. Lemma~\ref{lem:C7-cover-bridge} proves
that every affine rational point yields a point on one of the six
quartics through a norm map. \\[3pt]
Proposition~\ref{prop:C7-rational-points}; E
& Mordell--Weil rank \(2\); \texttt{Chabauty}: \(N=\#V=6\), \(R=1\)
& Finite index, rank \(2<4\), and curve/map good reduction above \(5\)
are checked. Since \(R=1\), the bound applies to the full group
\cite{magmaHandbookChabauty}. All six images and pullbacks are asserted. \\
\hline
\end{tabular}
\end{table}

The order-\(7\) jobs run in six fresh processes with seed \(1\).
Their local checks record two primes above \(5\), each of residue
degree \(2\) and ramification index \(1\). At both primes, the chosen
elliptic discriminant and the quartic leading coefficient and
discriminant are units; the actual reduced map and its infinity fiber
have degree \(2\). The note \path{LOCAL_REDUCTION.md} in the cover
directory explains the smooth models and transport of the map.

The two order-\(5\) rank certificates use the deficient-place
criterion and an explicit computation of the associated Cassels--Tate
pairing. They provide corroborating calculations and share the underlying
\(2\)-descent infrastructure; they are not independent software
implementations. For the explicit calculation, the Fisher--Yan pairing
routine, with Fisher--Liu minimization
\cite{fisherYanCasselsTate,fisherLiuMinimisation}, returns
\[
\begin{pmatrix}1&1\\1&1\end{pmatrix}
\quad\text{over }\mathbf F_2.
\]
This matrix has rank \(1\) and radical dimension \(1\); the pairing need
not be alternating. The rational Kummer image lies in this radical,
whereas its dimension is \(\rank J(P_5)(\Q)+1\). Thus the explicit
pairing also gives rank zero, without assuming finiteness of Sha
\cite[Section~3]{fisherYanCasselsTate}. The pinned sources, license, and
checksums are included. No script enables GRH bounds.

The remaining scripts check exact group-law and division-polynomial
identities, the modular \(j\)-maps, the order-\(5\) normalization and Prym
splitting, finite-field orders, and the fixed locus. They also verify the
order-\(7\) Tate \(c_4\)-identity and final fourth-power obstruction. The
order-\(4\) calculations corroborate the self-contained descent in
Appendix~\ref{app:C4-descent}, which proves Theorem~\ref{thm:C4-points}; bounded point searches are never used for
completeness. No proof conclusion rests on an uncertified numerical
comparison.

\FloatBarrier
\section{Arithmetic of the Quartic in the 4-Torsion Argument}\label{app:C4-descent}

This appendix supplies the explicit model calculation and elementary
$2$-descent used in Theorem~\ref{thm:C4-points}. The argument is independent
of the computer-assisted rank calculations elsewhere in the paper.

\begin{proposition}\label{prop:J4-arithmetic}
The maps \eqref{eq:C4-to-J4} extend to a $\Q$-isomorphism from the smooth
projective curve $C_4$ to
\[
J_4:\quad y^2=(x+16)(x^2-192),
\]
sending $(0,4)$ to $\mathcal O$ and $(0,-4)$ to $(-16,0)$. Moreover,
\[
J_4(\Q)=\{\mathcal O,(-16,0)\}.
\]
\end{proposition}

\begin{proof}
The quartic $3u^4+16u^2+16$ has discriminant $2^{20}\cdot3\neq0$, so $C_4$ is a curve of genus $1$; its two places at infinity are interchanged by $\operatorname{Gal}(\Q(\sqrt3)/\Q)$, because the leading coefficient $3$ is not a rational square, so they are not rational. Since $C_4$ contains the rational point $(0,4)$, it is an elliptic curve.

\emph{The Weierstrass model.} On $C_4$ one has the identity
\[
(v+4)^2-3u^4=v^2+8v+16-3u^4=8v+16u^2+32=8\,(v+4+2u^2),
\]
using $v^2=3u^4+16u^2+16$. Consequently, with $x$ as in \eqref{eq:C4-to-J4},
\begin{align*}
x+16&=\frac{8(v+4+2u^2)}{u^2},\\
x^2-192&=\frac{64\bigl[(v+4)^2-3u^4\bigr]}{u^4}\\
&=\frac{512\,(v+4+2u^2)}{u^4},
\end{align*}
and therefore
\begin{align*}
(x+16)(x^2-192)
&=\frac{4096\,(v+4+2u^2)^2}{u^6}\\
&=\Bigl(\frac{8(x+16)}{u}\Bigr)^{\!2}=y^2,
\end{align*}
as claimed; the displayed inverse is immediate from \eqref{eq:C4-to-J4}. A birational map between smooth projective curves extends to an isomorphism, so $C_4\cong J_4$ over $\Q$ and $\#C_4(\Q)=\#J_4(\Q)$. Under \eqref{eq:C4-to-J4} the point $(0,4)$ maps to the point at infinity, and $(0,-4)$ maps to $(-16,0)$: along the branch of $C_4$ through $(0,-4)$ one has $v+4=-2u^2+\tfrac18u^4+O(u^6)$, so $x+16=u^2+O(u^4)\to0$ and $y=8(x+16)/u\to0$.

\emph{The Mordell--Weil group of $J_4$.} The translation $x=W-16$ puts $J_4$ in the form
\[
J_4:\quad y^2=W\,(W^2-32W+64).
\]
The quadratic factor has discriminant $32^2-4\cdot64=768=2^8\cdot3\notin\Q^{\times2}$, so $T=(0,0)$ is the unique rational point of order $2$.

We first show that $J_4$ has rank $0$. By the standard $2$-isogeny construction \cite[Example~III.4.5]{silverman2009arithmetic}, the quotient of $J_4$ by $\langle T\rangle$ is
\[
y^2=W\bigl(W^2+64W+768\bigr)=W\,(W+16)(W+48),
\]
and the substitution $W=16x$, $y=64z$ identifies this curve with
\[
E'':\quad z^2=x\,(x+1)(x+3).
\]
Since Mordell--Weil rank is invariant under isogeny, it suffices to prove that $\operatorname{rank}E''(\Q)=0$, which we do by a complete $2$-descent \cite[Proposition~X.1.4]{silverman2009arithmetic}. Write $e_1=0$, $e_2=-1$, $e_3=-3$, and let
\[
\delta:E''(\Q)/2E''(\Q)\hookrightarrow \Q^\times/\Q^{\times2}\times\Q^\times/\Q^{\times2},\qquad
(x,z)\longmapsto (x-e_1,\;x-e_2),
\]
where \(\delta(\mathcal O)=(1,1)\). At \((e_i,0)\), a vanishing
coordinate \(x-e_i\) is replaced by
\(\prod_{k\neq i}(e_i-e_k)\); every other coordinate is evaluated
normally, and all entries are taken modulo squares.
The first coordinate of the image is supported on the primes dividing $(e_1-e_2)(e_1-e_3)=3$, and the second on those dividing $(e_2-e_1)(e_2-e_3)=-2$; thus the image lies in $\{\pm1,\pm3\}\times\{\pm1,\pm2\}$. Moreover $x(x+1)(x+3)\geq0$ forces $x\in[-3,-1]\cup[0,\infty)$ on $E''(\mathbb{R})$, so the two coordinates have equal signs. This leaves eight candidate classes. The images of the $2$-torsion points are
\[
\delta(0,0)=(3,1),\qquad \delta(-1,0)=(-1,-2),\qquad \delta(-3,0)=(-3,-2),
\]
which together with $(1,1)$ form a subgroup $H$ of order $4$; the remaining four candidates form the single coset $(1,2)H$. If the class $(b_1,b_2)=(1,2)$ lay in the $2$-Selmer group, the associated torsor equations
\[
z_1^2-2z_2^2=-1,\qquad z_1^2-2z_3^2=-3
\]
would be everywhere locally soluble. But the second has no solution in $\Q_3$: a primitive $\Z_3$-triple satisfying $p^2-2r^2+3q^2=0$ would give $p^2\equiv2r^2\pmod 3$, and since $2$ is not a square modulo $3$ this forces $3\mid p$ and $3\mid r$, hence $9\mid 3q^2$ and $3\mid q$, contradicting primitivity. As the Selmer group is a group containing $H$, the entire coset $(1,2)H$ is excluded. Hence the image of $\delta$ is exactly $H$, so $\#\,E''(\Q)/2E''(\Q)=4=\#\,E''(\Q)[2]$ and $\operatorname{rank}E''(\Q)=0$. Therefore $\operatorname{rank}J_4(\Q)=0$.

Next, the torsion. The discriminant of $J_4$ is supported on $\{2,3\}$, so $J_4$ has good reduction at $5$ and $7$, with
\[
\#\widetilde J_4(\mathbb{F}_5)=4,\qquad \#\widetilde J_4(\mathbb{F}_7)=12,
\]
by direct count. For every prime $\ell\neq5$ the $\ell$-primary part of $J_4(\Q)_{\mathrm{tors}}$ injects into $\widetilde J_4(\mathbb{F}_5)$, and the $5$-primary part injects into $\widetilde J_4(\mathbb{F}_7)$; hence $\#J_4(\Q)_{\mathrm{tors}}$ divides $4$. Finally, $T=(0,0)$ is not divisible by $2$ in $J_4(\Q)$: in Lemma~\ref{lem:halving} we have $B=64=8^2$, and $A+2\beta=-16$, $A-2\beta=-48$ are both negative, hence not squares. So there is no rational point of order $4$, and
\[
J_4(\Q)=J_4(\Q)_{\mathrm{tors}}=\{\mathcal O,\;T\}.
\]

In the original $x$-coordinate, the point $T$ is $(-16,0)$, as asserted.
\end{proof}

\FloatBarrier

\end{document}